\documentclass[12pt,a4paper]{article}
\usepackage{amssymb,amsmath,amsthm}
\usepackage[english]{babel}
\usepackage{t1enc}
\usepackage[latin2]{inputenc}
\usepackage{epsfig}
\usepackage{subfig}
\usepackage{epstopdf}
\usepackage{xcolor}
\usepackage{bookmark}

\newtheorem{thm}{Theorem}
\newtheorem{cor}[thm]{Corollary}
\newtheorem{defi}[thm]{Definition}
\newtheorem{lem}[thm]{Lemma}

\newtheorem{claim}[thm]{Claim}

\newtheorem{obs}[thm]{Observation}

\theoremstyle{remark}

\usepackage{indentfirst}
\def\HH{\mathcal {H}}
\def\RR{\mathcal {R}}
\def\cH{\mathcal {H}}
\def\CC{\mathcal {C}}
\def\FF{\mathcal {F}}
\def\AA{\mathcal {A}}
\def\TT{\mathcal {T}}

\def\BB{\mathcal {B}}
\def\cD{\mathcal {D}}
\def\red1{\color{red}1\color{black}}
\def\blue1{\color{blue}1\color{black}}

\makeatletter
\def\blfootnote{\gdef\@thefnmark{}\@footnotetext}
\makeatother

\begin{document}

\title{Discrete Helly-type theorems for pseudohalfplanes}

\author{Bal\'azs Keszegh\thanks{Alfréd Rényi Institute of Mathematics and MTA-ELTE Lendület Combinatorial Geometry Research Group, Budapest, Hungary. Research supported by the Lend\"ulet program of the Hungarian Academy of Sciences (MTA), under the grant LP2017-19/2017 and by the National Research, Development and Innovation Office -- NKFIH under the grant K 132696.}
}

\maketitle


\begin{abstract}
We prove discrete Helly-type theorems for pseudohalfplanes, which extend recent results of Jensen, Joshi and Ray about halfplanes. Among others we show that given a family of pseudohalfplanes $\HH$ and a set of points $P$, if every triple of pseudohalfplanes has a common point in $P$ then there exists a set of at most two points that hits every pseudohalfplane of $\HH$. We also prove that if every triple of points of $P$ is contained in a pseudohalfplane of $\HH$ then there are two pseudohalfplanes of $\HH$ that cover all points of $P$.

To prove our results we regard pseudohalfplane hypergraphs, define their extremal vertices and show that these behave in many ways as points on the boundary of the convex hull of a set of points. Our methods are purely combinatorial.

In addition we determine the maximal possible chromatic number of the regarded hypergraph families.
\end{abstract}

\section{Introduction}\label{sec:intro}

Given a (finite) point set $P$ and a family of regions $\RR$ (e.g., the family of all halfplanes) in the plane (or in higher dimensions), let $\HH$ be the hypergraph with vertex set $P$ and for each region of $\RR$ having a hyperedge containing exactly the same points of $P$ as this region. There are many interesting problems that can be phrased as a problem about hypergraphs defined this way, which are usually referred to as \emph{geometric hypergraphs}. This topic has a wide literature, researchers considered problems where $\RR$ is a family of halfplanes, axis-parallel rectangles, translates or homothets of disks, squares, convex polygons, pseudo-disks and so on. There are many results and open problems about the maximum number of hyperedges of such a hypergraph, coloring questions and other properties. For a survey of some of the most recent results see the introduction of the paper of Ackerman, Keszegh and Pálvölgyi \cite{abab} and of the paper of Damásdi and Pálvölgyi \cite{dp}, for an up-to-date database of such results with references see the webpage \cite{cogezoo}.

One of the most basic families is the family of halfplanes, about which already many problems are non-trivial. Among others one such problem was considered by Smorodinsky and Yuditsky \cite{smorodinsky-yuditsky} where they prove that the vertices of every hypergraph defined by halfplanes on a set of points (i.e., $P$ is a finite set of points and $\RR$ is the family of all halfplanes) can be $k$-colored such that every hyperedge of size at least $2k+1$ contains all colors. Keszegh and Pálvölgyi \cite{abafree} considered generalizing this result by replacing halfplanes with the family of translates of an unbounded convex region (e.g., an upwards parabola). It turned out that this is true even when halfplanes are replaced by \emph{pseudohalfplanes}. The main tool of proving this was an equivalent combinatorial definition of so called \emph{pseudohalfplane hypergraphs}, hypergraphs that can be defined on points with respect to pseudohalfplanes.\footnote{The exact definitions of pseudohalfplanes and pseudohalfplane hypergraphs are postponed to Section \ref{sec:pshpdef}.} This formulation had the promise that many other statements about halfplane hypergraphs can be generalized to pseudohalfplane hypergraphs in the future. While this combinatorial formulation has the disadvantage of being less visual and thus somehow less intuitive than the geometric setting, it has many advantages, among others covering a much wider range of hypergraphs, also, being purely combinatorial might have algorithmic applications as well. One recent application is a similar polychromatic coloring result of Damásdi and Pálvölgyi \cite{dp} about disks all containing the origin where after observing that in every quadrant of the plane the disks form a family of pseudohalfplanes they can apply the results from \cite{abafree}. 

In \cite{abafree} the equivalent of the convex hull vertices in the plane (more precisely, the points on the boundary of the convex hull) was defined for pseudohalfplane hypergraphs and called \emph{unskippable vertices} and this made it possible to generalize the proof idea of \cite{smorodinsky-yuditsky} from halfplanes to pseudohalfplane hypergraphs. To ease intuition, we call unskippable vertices as \emph{extremal vertices} from here on. Exact definitions of these notions are postponed.

Recently Jensen, Joshi and Ray \cite{jjr} proved discrete Helly-type theorems which can be formulated in terms of halfplane hypergraphs, their results are detailed in Section \ref{sec:hp}. In this paper we generalize their results to pseudohalfplane hypergraphs, in addition we also prove one missing variant for which even the halfplane counterpart was not considered yet. Again we make use of extremal vertices defined in \cite{abafree}, but we need to prove many new properties of extremal vertices which show that extremal vertices behave in many ways as convex hull vertices in the plane (more precisely, as the points on the boundary of the convex hull). We believe that these properties will be useful also for future research on pseudohalfplane hypergraphs. We also consider these problems for pseudohemisphere hypergraphs, a natural hypergraph family containing the family of pseudohalfplane hypergraphs.

We consider the following two types of problems: in a \emph{primal discrete Helly theorem of type $k\rightarrow l$} let $P$ be a set of $n$ points (resp. vertex set) and $\FF$ be a family of regions (resp. hypergraph). If every $k$-tuple of regions (resp. hyperedges) in $\FF$ intersects at a point (resp. vertex) in $P$, then there exists a set of $l$ points (resp. vertices) in $P$ that intersects each $F\in \FF$. In a \emph{dual discrete Helly theorem of type $k\rightarrow l$} let $P$ be a finite set of $n$ points (resp. vertices) and $\FF$ be a family of regions (resp. hypergraph). If every subset of $k$ points in $P$ belongs to some region (resp. hyperedge) $F\in \FF$ then there exist $l$ regions (resp. hyperedges) in $\FF$ whose union covers $P$.

In Table \ref{tab:summary} we summarize our results. For all our results we show that they are optimal except for the ones about pseudohemispheres.

\begin{table}[]
	\centering
	\begin{tabular}{|c|c|c|c|c|c|}
		 \hline
		 \multicolumn{2}{|c|}{halfplane}                                                     & ABA-free                                & \multicolumn{2}{|c|}{pseudohalfplane}    
		& pseudohemisphere 
		\\        primal                         &  dual                            &      primal/dual                     &    primal                     &   dual                 
		& primal/dual
		\\
		\hline
		$3\rightarrow 2$\cite{jjr} & $3\rightarrow 2$\cite{jjr} & $2\rightarrow 2$ & $3\rightarrow 2$ & $3\rightarrow 2$ 
		& $4\rightarrow 2$ 
		\\
		(Thm \ref{thm:primalhp}) & 		(Thm \ref{thm:dualhp1}) & (Thm \ref{thm:primalaba}, Cor \ref{cor:dualaba})& (Thm \ref{thm:primalpshp32}) & (Thm \ref{thm:dualpseudohp32}) 
		& (Thm \ref{thm:primalpseudohemi}, Thm \ref{thm:dualpseudohsp24}) 		
		\\$2\rightarrow 3$& $2\rightarrow 3$\cite{jjr}      && $2\rightarrow 3$                                        & $2\rightarrow 3$ 
		&      
		\\
		(Thm \ref{thm:primalpshp}) & 		(Thm \ref{thm:dualhp2}) &  & (Thm \ref{thm:primalpshp}) & (Thm \ref{thm:dualpseudohp23}) 
		&\\ 	
		\hline	                                         
	\end{tabular}
	\caption{Summary of the considered Helly-type results}
	\label{tab:summary}
\end{table}

In order to show that our primal and dual results about pseudohalfplanes could not be handled together we show that the chromatic number differentiates the primal and dual setting. In order to do that we prove that the maximal possible chromatic number of pseudohalfplane hypergraphs is $4$ while the maximal possible chromatic number of duals of pseudohalfplane hypergraphs is $3$.

As mentioned, such discrete Helly-type problems were considered earlier by Jensen et al. \cite{jjr} for halfplanes. We are aware of only two further papers considering such problems. First, Halman \cite{halman} among others proved discrete Helly-type results about axis-parallel boxes. Second, while it is easy to see that in general a discrete Helly-type theorem for convex sets is not true (see the example at the beginning of Section \ref{sec:hp}), yet an old result of Doignon \cite{doignon} states that given a finite family $\cal H$ of convex sets in $R^d$, if every $2^d$ or fewer members of $\cal H$ have a common point with integer coordinates, then there is a point with integer coordinates common to all members of $\cal H$.

The paper is structured as follows. First in Section \ref{sec:pshpdef} we define pseudohalfplane hypergraphs, the objects we study. In Section \ref{sec:hp} we give an account of the discrete Helly-type results of Jensen et al. \cite{jjr} which we generalize to pseudohalfplanes in Section \ref{sec:pshp}, these results are proved in Section \ref{sec:proofpshp} using properties of extremal vertices proved in Section \ref{sec:extremal}. In Section \ref{sec:vs} we discuss why and how much our setting is more general than the usual geometric setting of halfplanes. In Section \ref{sec:chromatic} we state and prove our results about proper coloring pseudohalfplane and dual pseudohalfplane hypergraphs. Finally, in Section \ref{sec:discussion} we give some directions for further research.

\subsection{Pseudohalfplanes and pseudohalfplane hypergraphs}\label{sec:pshpdef}

\textbf{Pseudohalfplane hypergraphs.} The definition of \emph{pseudohalfplane hypergraphs} introduced in \cite{abafree} is based on the definition of ABA-free hypergraphs and is as follows.

\begin{defi}\label{def:ABA}
	A hypergraph $\mathcal H$ with an ordered vertex set is called {\em ABA-free} if $\mathcal H$ does not contain two hyperedges $A$ and $B$ for which there are three vertices $x<y<z$ such that $x,z\in A\setminus B$ and $y\in B\setminus A$.\footnote{We imagine the vertices on a horizontal line, and thus if $x<y$ then we may say that $x$ is to the left from $y$ and so on.}
\end{defi}

\begin{defi}\label{def:pseudo}
	A hypergraph $\HH$ on an ordered set of vertices $V$ is called a  {\em pseudo\-halfplane hypergraph} if there exists an ABA-free hypergraph $\FF$ on $V$ such that $\HH\subset\FF\cup \bar{\FF}$.\footnote{$\bar{\FF}$ denotes the family of the complements of the hyperedges of $\FF$. It was shown in \cite{abafree} that  $\bar\FF$ is also ABA-free if $\FF$ is ABA-free.}
\end{defi}

\textbf{Pseudolines.}
A {\em loose pseudoline arrangement} is a finite collection of simple curves in the plane such that each curve cuts the plane into two unbounded components (i.e., both endpoints of each curve are at infinity) and any pair of curves is either disjoint or intersects once and in the intersection point the two curves cross. 
A \emph{pseudoline arrangement} is a loose pseudoline arrangement in which no two curves are disjoint (and so they cross once).\footnote{Pseudoline arragements are usually defined in the projective plane, as a collection of simple closed curves whose removal does not disconnect the projective plane and for which every pair of the curves intersects no more than once (hence they intersect exactly once where they cross). However, in the literature sometimes pseudoline arrangements are defined as we now defined loose pseudoline arrangements. We differentiate between these two notions to avoid confusion and also to make clear that most of our results apply to the more general case of loose pseudoline arrangements.}
 A (loose) arrangement of pseudolines is \emph{simple} if no three pseudolines meet at a point. Wlog. we can assume that the pseudolines are $x$-monotone bi-infinite curves (see, e.g. \cite{abafree}), such arrangements are sometimes called Euclidean or \emph{graphic} pseudoline arrangements. For an introduction into pseudoline arrangements see Chapter 5 of \cite{handbook} by Felsner and Goodman.

\textbf{Pseudohalfplanes.}
Given a pseudoline arrangement, a \emph{pseudohalfplane family} is the subfamily of the above defined components (one on each side of each pseudoline). A pseudohalfplane family is simple (resp. loose) if the boundaries form a simple (resp. loose) pseudoline arrangement.
A pseudohalfplane family is upwards if we just take components that are above the respective pseudoline (here we use that the pseudolines are assumed to be $x$-monotone).


In \cite{abafree} it is shown that given a loose family $\FF$ of pseudohalfplanes in the plane and a set of points $P$ then the hypergraph whose hyperedges are the subsets that we get by intersecting regions of $\FF$ with $P$ is a pseudohalfplane hypergraph, and conversely, every pseudohalfplane hypergraph can be realized this way with a (simple and not loose) family of pseudohalfplanes.\footnote{In fact they prove that we can realize them with simple loose pseudoline arrangements but their argument can be easily modified to have a realization with a simple and not loose pseudoline arrangement as well.} If $\FF$ is a family of upwards pseudohalfplanes then we get the ABA-free hypergraphs. Thus, all our results about pseudohalfplane hypergraphs implies the respective result about (loose and not loose) families of pseudohalfplanes where we replace vertices with points and hyperedges with pseudohalfplanes. For the same reason, slightly abusing our notation, we may refer to the hyperedges of a pseudohalfplane hypergraph as pseudohalfplanes.

\subsection{Helly-type theorems for halfplanes}\label{sec:hp}

Helly's classic theorem in the plane \cite{Helly1923} can be phrased as follows:
\begin{thm}[Helly for convex sets]
	Let $P$ be a set of $n$ points and $\CC$ be a finite family of convex sets in the plane. If every subfamily of $3$ convex sets from $\CC$ intersects in a point of $P$ then there exists a point (not necessarily in $P$) which is in every convex set of $\CC$.
\end{thm}

Halman and Jensen et al. \cite{halman,jjr} regarded discrete versions of Helly's theorem, where they require that the point one finds also comes from the set $P$. First, the following simple construction \cite{halman,jjr} shows that we cannot require this for convex sets, even if we replace $3$ by some larger value $k$ and we want to find only some bounded number of vertices that hit all sets: take a set $P$ of $n$ points in convex position, then every subset of points in $P$ can be separated from the rest of the points in $P$ by a convex set. Now for some fixed $k$ let $\CC$ be the family of such separating convex sets for the subsets of points in $P$ of size more than $n-n/k$. Then every subfamily of size $k$ of $\CC$ has a common point in $P$, on the other hand no subset of points in $P$ of size less than $n/k$ hits every set in $\CC$.

They show that replacing convex sets with halfplanes yields interesting problems and prove the following results:

\begin{thm}[Dual Discrete Helly for halfplanes, $3\rightarrow 2$]\label{thm:dualhp1}\cite{jjr}
	Let $P$ be a set of $n$ points and $\cH$ be a family of halfplanes. If every subset of 3 points in $P$ belongs to some halfplane $H\in \cH$ then there exist two halfplanes in $\cH$ whose union covers $P$.
\end{thm}

They give an example that this is tight, that is, $3$ cannot be replaced by $2$. They also show the following:

\begin{thm}[Dual Discrete Helly for halfplanes, $2\rightarrow 3$]\cite{jjr}\label{thm:dualhp2}
	Let $P$ be a set of $n$ points and $\cH$ be a family of halfplanes. If every pair of points in $P$ belongs to some halfplane $H\in \cH$ then  there exists $3$ halfplanes in $\cH$ whose union covers $P$.
\end{thm}

\begin{thm}[Primal Discrete Helly for halfplanes, $3\rightarrow 2$]\cite{jjr}\label{thm:primalhp}
	Let $P$ be a set of $n$ points and $\cH$ be a family of halfplanes. If every triple of halfplanes in $\cH$ intersects at a point in $P$, then there exists a set of two points in $P$ which intersects each $H\in \cH$.
\end{thm}

The above two results are implied by their following result about convex pseudodisks:

\begin{thm}[Primal Discrete Helly for convex pseudodisks, $3\rightarrow 2$]\cite{jjr}\label{thm:convpseudodisks}
	Let $P$ be a set of $n$ points and $\cD$ be a family of convex pseudodisks. If every triple of pseudodisks in $\cD$ intersects at a point in $P$, then there exists a set of two points in $P$ which intersects each $D\in \cD$.
\end{thm}

\subsection{Helly-type theorems for pseudohalfplanes}\label{sec:pshp}

We aim to prove results about pseudohalfplanes similar to these about halfplanes from the previous section. First we show discrete Helly-type results for ABA-free hypergraphs:

\begin{thm}[Primal Discrete Helly for ABA-free hypergraphs, $2\rightarrow 2$]\label{thm:primalaba}
	Given an ABA-free $\HH$ such that every pair of hyperedges has a common vertex, there exists a set of at most two vertices that hits every hyperedge of $\HH$.
\end{thm}

As the dual of an ABA-free hypergraph is also an ABA-free hypergraph \cite{abafree}, this implies (and is in fact equivalent to):
\begin{cor}[Dual Discrete Helly for ABA-free hypergraphs, $2\rightarrow 2$]\label{cor:dualaba}
	Given an ABA-free $\HH$ on vertex set $V$ of size $n\ge 2$ such that for every pair of vertices there is a hyperedge of $\HH$ containing both of them, there exists at most two hyperedges of $\HH$ whose union covers $V$.
\end{cor}

Applying this twice to the two ABA-free parts of a pseudohalfplane hypergraph implies easily that $2\rightarrow 4$ is true for pseudohalfplanes but we can prove a better bound which is optimal (we note that this was not known earlier even in the special case of halfplanes):

\begin{thm}[Primal Discrete Helly for pseudohalfplanes, $2\rightarrow 3$]\label{thm:primalpshp}
	Given a pseudohalfplane hypergraph $\HH$ such that every pair of hyperedges has a common vertex, there exists a set of at most $3$ vertices that hits every hyperedge of $\HH$.
\end{thm}

We can also prove the following:

\begin{thm}[Primal Discrete Helly for pseudohalfplanes, $3\rightarrow 2$]\label{thm:primalpshp32}
		Given a pseudohalfplane hypergraph $\HH$ such that every triple of hyperedges has a common vertex, there exists a set of at most $2$ vertices that hits every hyperedge of $\HH$.
\end{thm}

In the dual setting we have the following results about pseudohalfplanes:

\begin{thm}[Dual Discrete Helly for pseudohalfplanes, $3\rightarrow 2$]\label{thm:dualpseudohp32}
	Given a pseudohalfplane hypergraph $\HH$ on ordered vertex set $V$ with $n\ge 3$ vertices. If every subset of $3$ vertices in $V$ is contained by some hyperedge $H\in \cH$ then there exist at most two hyperedges in $\cH$ whose union covers $V$.
\end{thm}

\begin{thm}[Dual Discrete Helly for pseudohalfplanes, $2\rightarrow 3$]\label{thm:dualpseudohp23}
	Given a pseudohalfplane hypergraph $\HH$ on ordered vertex set $V$ with $n\ge 2$ vertices. If every pair of vertices in $V$ is contained by hyperedge $H\in \cH$ then there exist at most $3$ hyperedges in $\cH$ whose union covers $V$.
\end{thm}

We can show a similar result about pseudohemisphere hypergraphs, which generalize both pseudohalfplane hypergraphs and duals of pseudohalfplane hypergraphs.\footnote{The dual of a hypergraph is the hypergraph we get by exchanging the roles of hyperedges and vertices while reversing the containment relation.}

\begin{defi}\label{def:signedpseudo}\cite{abafree}
	A {\em pseudohemisphere hypergraph} is a hypergraph $\HH$ on an ordered set of vertices $V$ such that there exists a set $X\subset V$ and an ABA-free hypergraph $\cal F$ on $V$ such that the hyperedges of $\HH$ are some subset of $\{F\Delta X, \bar F \Delta X \mid F\in \cal F\}$.
\end{defi}

\begin{thm}[Dual Discrete Helly for pseudohemispheres, $2\rightarrow 4$]\label{thm:dualpseudohsp24}
	Given a pseudohemisphere hypergraph $\HH$ on ordered vertex set $V$ with $n\ge 2$ vertices. If every pair of vertices in $V$ is contained by some hyperedge $H\in \cH$ then there exist at most $4$ hyperedges in $\cH$ whose union covers $V$.
\end{thm}

As the dual of a pseudohemisphere hypergraph is also a pseudohemisphere hypergraph \cite{abafree}, this also implies:

\begin{thm}[Primal Discrete Helly for pseudohemispheres, $2\rightarrow 4$]\label{thm:primalpseudohemi}
	Given a pseudohemisphere hypergraph $\HH$ such that every pair of hyperedges has a common vertex, there exists a set of at most four vertices that hits every hyperedge of $\HH$.
\end{thm}

\section{Properties of the extremal vertices}\label{sec:extremal}

First we recall and prove some properties of ABA-free and pseudohalfplane hypergraphs.

\begin{defi}
	In an ABA-free hypergraph $\FF$, a vertex $a$ is {\em skippable} if there exists an $A\in \FF$ such that $\min(A)< a < \max(A)$ and $a\notin A$.
	In this case we say that $A$ {\em skips} $a$. 
	A vertex $a$ is {\em unskippable} if there is no such $A$.
\end{defi}

\begin{lem}\label{lem:unskippable}\cite{abafree}
	If $\mathcal F$ is ABA-free, then every $A\in \mathcal F$ contains an unskippable vertex.
\end{lem} 

Observe that by definition the leftmost (that is, first) and rightmost (that is, last) vertex is unskippable. Thus for every skippable vertex $v$ there exists a closest unskippable vertex after and before $v$.

\begin{lem}\label{lem:assocunskippable}
	If $\mathcal F$ is an ABA-free hypergraph on vertex set $V$ and $v\in V$ is skippable, then every hyperedge $H$ which contains $v$ must contain at least one of the two unskippable vertices before and after $v$ that are closest to $v$.
\end{lem}

\begin{proof}
	Assume on the contrary. Let $l$ (resp. $r$) be the closest unskippable vertex to $v$ left to $v$ (resp. right to $v$). By Lemma \ref{lem:unskippable} $H$ contains some unskippable vertex $w$ different from $l$ and $r$. If $w$ is left to $l$ (resp. right to $r$) then $H$ skips $l$, contradicting that $l$ (resp. $r$) is unskippable. Thus $l<w<r$ in the vertex order, contradicting that $l$ and $r$ were the closest unskippable vertices to $v$.
\end{proof}

\begin{lem}\label{lem:abafree-addvertex}
	Given an ABA-free hypergraph $\FF$ on vertex set $V$ and a vertex $w$ of $\FF$. Let $\FF'$ be the subhypergraph of $\FF$ induced by the vertex set $V\setminus \{w\}$\footnote{Given a hypergraph $\FF$ on vertex set $V$, the subhypergraph induced by a subset $V'\subset V$ is the hypergraph on vertex set $V'$ with hyperedge set $\{F':F'\subseteq V'$ and $\exists F\in \FF$ s.t. $F'=F\cap V'\}$.}. Let $v$ be an unskippable vertex of $\FF'$, then at least one of $v$ and $w$ is unskippable in $\FF$.
\end{lem}
\begin{proof}
	Wlog. suppose that $v<w$. Suppose that in $\FF$ there is a hyperedge $H$ that skips $v$ and a $H'$ hyperedge that skips $w$. 
	Thus $H$ does not contain $v$ and as $v$ is unskippable in $\FF'$, $H$ must contain $w$ and no other vertex bigger than $v$.
	
	Also, $H'$ contains a vertex $q$ bigger than $w$. If $H'$ would contain $v$ then $H$ and $H'$ would form an ABA occurrence on the vertices $v,w,q$. If $H'$ would contain a vertex $r$ betwen $v$ and $w$ then $H$ and $H'$ would form an ABA occurrence on the vertices $r,w,q$ (note that $r\ne w$ is not in $H$ as it is bigger than $v$). Thus $H'$ does not contain $v$ nor a vertex between $v$ and $w$. If $H'$ would contain a vertex smaller than $v$ then $H'$ would skip $v$ in $\FF'$, contradicting our assumption. Altogether, we have shown that $H'$ cannot contain any vertex smaller than $w$ contradicting that it skips $w$.
\end{proof}

Now we extend the definition of unskippable vertices to pseudohalfplane hypergraphs as in \cite{abafree} and call them extremal vertices:
\begin{defi}
	Given a pseudohalfplane hypergraph $\HH$ such that $\HH\subseteq \FF\cup \bar{\FF}$ for some ABA-free hypergraph $\FF$. Call $\TT=\HH\cap \FF$ the topsets and $\BB=\HH\cap \bar{\FF}$ the bottomsets, observe that both $\TT$ and $\BB$ are ABA-free. The unskippable vertices of $\FF$ (resp.\ $\bar{\FF}$) are called topvertices (resp.\ bottomvertices).\footnote{Notice  that the top and bottom vertices depend on $\FF$ and not on $\HH$ itself (and thus without fixing $\FF$ we cannot directly talk about the extremal vertices of $\HH$). For a given $\HH=\TT\cup\BB$ multiple $\FF$'s can witness that it is a pseudohalfplane hypergraph, the smallest valid family is $\TT\cup\bar{\BB}$. Although it is not assumed, yet when reading the paper it is convenient to assume that $\HH=\FF\cup \bar \FF$.} The union of the topvertices and bottomvertices is called the set of \emph{extremal vertices} of $\HH$ and is denoted by $C(\HH).$\footnote{$C(\HH)$ is sometimes abbreviated to $C$ when the underlying hypergraph is clear from the context.} 	
\end{defi}

In the remainder of this section we are always given a pseudohalfplane hypergraph $\cH$ on vertex set $V$ whose extremal vertices are denoted by $C$. First, the following observation provides an equivalent definition of extremal vertices:

\begin{claim}\label{claim:singleton}
	The topvertices are exactly those vertices $v$ for which if we add the singleton hyperedge $\{v\}$ to $\FF$ we still get a pseudohalfplane hypergraph. The bottomvertices are exactly those vertices $v$ for which if we add the singleton hyperedge $\{v\}$ to $\bar{\FF}$ (that is, we add $V\setminus \{v\}$ to $\FF$) we still get a pseudohalfplane hypergraph. 
	
	In other words, the extremal vertices are exactly those that can be separated from the rest of the vertices by a (possibly additional) pseudohalfplane.\footnote{Note that with halfplanes in the plane this wording would give the extreme vertices instead of the vertices that lie on the boundary of the convex hull.}
\end{claim} 
\begin{proof}
	This follows easily from the definition of unskippability. First, $v$ being a topvertex is by definition equivalent to having no hyperedge in $\FF$ that contains a vertex before and after $v$ but does not contain $v$ which is equivalent to that adding $\{v\}$ to $\FF$ does not introduce two hyperedges in $\FF$ that form an ABA occurrence. 
	
	The part about bottomvertices follows the same way.
\end{proof}

The following observation provides the intuition why we refer to $C$ as the extremal vertices.

\begin{claim}
	If $\HH$ is defined by halfplanes\footnote{That is, $\HH= \FF\cup \bar{\FF}$ where $\FF$ (resp. $\bar{\FF}$) is the family of sets that we get by intersecting all the upwards (resp. downwards) halfplanes with $P$. It is easy to see that this $\FF$ is indeed ABA-free, we refer to \cite{abafree} for further details.} on a point set $P$ then the set of extremal vertices $C(\HH)$ of $\HH$ coincides with the set of points of $P$ that lie on the boundary of the geometric convex hull of $P$.\footnote{Note that this is different from the set of \emph{extreme vertices} which usually denotes the set of vertices of the convex hull (which is a subset of the vertices that lie on the boundary of the convex hull). In case it is assumed that the points are in general position, extremal vertices and extreme vertices coincide and then we can think about the extremal vertices as the vertices of the convex hull, that's why we have chosen a very similar word. For more on this, see Section \ref{sec:extremalvsextreme}.} 
\end{claim}

\begin{proof}
	In the rest we refer to points as vertices when we deal with the abstract extremal vertices.
	By definition the topvertices of $\HH$ are exactly those vertices which are unskippable in $\FF$. The points on the geometric upper hull of $P$ have this property, as if a point $p$ is on the upper hull then any upwards halfplane that contains a point both to the left and to the right from $v$ must also contain $v$. On the other hand if a point $q$ is not on the boundary of the geometric upper hull then there is an edge of the hull that goes above $q$ and then the upwards halfplane which contains only the vertices on this convex hull edge skips $q$. Thus topvertices of $\HH$ are exactly the points on the upper hull of $P$.
	
	Similarly, the bottomvertices of $\HH$ are exactly the points on the lower hull of $P$, finishing the proof.
\end{proof}

Now we prove several properties of the (abstract) extremal vertex set which are all generalizations of well-known properties of the set of points that lie on the boundary of the geometric convex hull (we refer to these as the geometric extremal vertices). Most of these properties we will use later, but we also prove some which we do not use later but nevertheless think that they further our understanding of extremal vertices and may be useful in future research.

\begin{obs}
	The leftmost and rightmost vertices are both topvertices and bottomvertices and so they are always extremal vertices. 
\end{obs}

\begin{claim}\label{claim:unskippablepshp}
	Every pseudohalfplane contains an extremal vertex.
\end{claim}

\begin{proof}
	If $H$ is a topset then applying Lemma \ref{lem:unskippable} to $\FF$ we get that $H$ contains an unskippable vertex of $\FF$, which is a topvertex, and so is in $C$. 	If $H$ is a bottomset then applying Lemma \ref{lem:unskippable} to $\bar\FF$ we get that $H$ contains an unskippable vertex of $\bar\FF$, which is a bottomvertex, and so is in $C$.
\end{proof}

\begin{claim}
	If the vertex set has size $n\ge 3$, then the extremal vertex set contains at least $3$ vertices.
\end{claim}
\begin{proof}
We have seen that the leftmost and rightmost vertex is always in $C$.
Now if we have $n=3$ vertices altogether then it is easy to see that the middle vertex must be a top or bottom vertex or both. If we have $n>3$ vertices then if we delete an arbitrary vertex $w$ which is neither leftmost nor rightmost then by induction in the subhypergraph induced by $V\setminus \{w\}$ there is an extremal vertex $v$ which is neither leftmost nor rightmost, wlog. it is a topvertex. Applying Lemma \ref{lem:abafree-addvertex} we get that in $\FF$ at least one of $w$ and $v$ is a topvertex, thus part of the extremal vertex set, and we are done.
\end{proof}

The next statements are easy consequences of the definition of unskippability:

\begin{obs}[Topvertices in a topset]\label{obs:topintop} If $X$ is a topset and $x,y\in X$, then $X$ contains all topvertices that are between $x$ and $y$. The same holds with bottomvertices if $X$ is a bottomset.
\end{obs}

\begin{obs}[Bottomvertex in a topset]\label{obs:topinbottom} If $X$ is a topset and $x\in X$ is a bottomvertex, then $X$ contains all vertices that are bigger or all vertices that are smaller than $x$. The same holds if $X$ is a bottomset and $x\in X$ is a topvertex.
\end{obs}

Let $T=(t_1=v_1,t_2,\dots t_k=v_n)$ and $B=(b_1=v_1,b_2,\dots b_l=v_n)$ be the sets of top and bottom vertices ordered according to the ordering on $P$. Call $T$ to be the upper hull and $B$ the lower hull. Note that a vertex may appear in both sets. Let us give the following circular order on $C$, the set of extremal vertices: $C=(v_1,t_2,\dots, t_{k-1},v_n, b_{l-1},\dots, b_2)$\footnote{This circular order corresponds to the clockwise order of points on the boundary of the convex hull in the geometric case defined by halfplanes. Also, $T$ corresponds to the upper hull and $B$ to the lower hull in the geometric case.}.

\begin{lem}\label{lem:hullinterval}
	Every pseudohalfplane intersects the extremal vertex set in an interval of the circular order defined on the extremal vertex set.
\end{lem}
\begin{proof}
	By symmetry it is enough to prove the statement when the pseudohalfplane $H$ is a topset.
	
	If the pseudohalfplane hyperedge is empty then the claim trivially holds.
	Otherwise, by Observation \ref{obs:topinbottom} $H$ intersects $B$ in an interval that has $v_1$, $v_n$ or both as an endvertex. By Observation \ref{obs:topintop} $H$ intersects $T$ in an interval. As $v_1$ and $v_n$ are also endvertices of $T$, $v_1$, $v_n$ or both (whichever was in $H\cap B$) must be an endpoint of $H\cap T$. Together the two intervals $H\cap B$ and $H\cap T$ form $H\cap C$ which is thus an interval.
%
%
%
\end{proof}

\begin{lem}\label{lem:consecutives}
	If the pseudohalfplane $H\in \cH$ is a topset (resp. bottomset) and contains two bottomvertices (resp. topvertices) $p<q$ that are consecutive in the circular order of the extremal vertices, then $H$ contains every vertex $r$ with $p<r<q$.
\end{lem}
\begin{proof}
	Let $H$ be a topset containing the consecutive  bottomvertices $p<q$ (the other case is symmetrical). Applying Observation \ref{obs:topinbottom} to $H$ and $p$ and to $H$ and $q$ we get that either all vertices $p<r$ or all vertices $r<q$ are contained in $H$ and we are done or all vertices $r$ for which $r<p$ or $p<r$ holds are contained in $H$. We show that in this case $H$ actually contains every vertex. Suppose on the contrary that some vertex $v$, $p<v<q$ is not in $H$. Now $\bar H$ is a non-empty set in $\bar{\FF}$ that contains only vertices between $p$ and $q$. Using Lemma \ref{lem:unskippable} we get that $\bar H$ contains an unskippable vertex of $\bar\FF$, which by definition is a bottomvertex of $\cH$, and lies between $p$ and $q$, a contradiction (as $p$ and $q$ were consecutive in the circular order of the extremal vertices).
\end{proof}

\begin{claim}\label{claim:everyconsecutive}
	If a topset (resp. bottomset) $H\in \cH$ contains every bottomvertex (resp. topset) then it contains every vertex.
	
	If the pseudohalfplane $H\in \cH$ contains every extremal vertex then it contains every vertex.
\end{claim}
\begin{proof}
	Wlog. suppose $H$ is a topset. The other cases follow from this. Applying Lemma \ref{lem:consecutives} on every pair of consecutive bottomvertices ($b_1=v_1$ and $b_2$, $b_2$ and $b_3$,\dots, $b_{l-1}$ and $b_l=v_n$) we get that all vertices are in $H$.
\end{proof}

\begin{lem}\label{lem:twopshpcover}
	Given two extremal vertices $p$ and $q$, there are two intervals on the extremal vertex set's circular order that have these as endpoints. Suppose $H_1\in \cH$ contains one of these intervals and $H_2\in \cH$ contains the other, then $H_1\cup H_2$ contains every vertex.
\end{lem}
\begin{proof}
	Wlog. $p<q$. We can suppose that one of $H_1$ and $H_2$ is a topset and the other is a bottomset. Indeed, if they are of the same type, wlog. topsets, then we can apply Lemma \ref{lem:consecutives} on every pair of consecutive bottomvertices ($b_1=v_1$ and $b_2$, $b_2$ and $b_3$,\dots, $b_l-1$,$b_l=v_n$) with $H_1$ or $H_2$ to get that all vertices are in $H_1$ or $H_2$.

	Thus wlog. $H_1$ is a topset and $H_2$ is a bottomset.

	First, for every vertex $r$ with $p<r<q$, $r$ must be in $H_1$ or $H_2$. Indeed, otherwise $H_1$ and $\bar H_2$ are two hyperedges of $\FF$ forming an ABA-occurrence on $p,q,r$, contradicting that $\FF$ is ABA-free.
	
	Second, by the assumption of the lemma, for every vertex $r<p$ which is not an extremal vertex, either for $H=H_1$ or $H=H_2$ it is true that $H$ contains the topvertices and bottomvertices right before and after $r$ in the order of the vertices (altogether at most $4$ vertices). Depending on $H$ (if it is a topset or bottomset), we can apply \ref{claim:everyconsecutive} on two of these vertices (the two topvertices or the two bottomvertices) to conclude that $r$ is in $H$. Symmetrical argument shows that every vertex $r>q$ is also in $H_1$ or $H_2$, finishing the proof.	
\end{proof}

\begin{lem}\label{lem:twopshpdisjoint}
	Given two non-consecutive extremal vertices (in the circular order) $p$ and $q$, $C\setminus\{p,q\}$ is the union of two intervals of the circular order of $C$. Suppose $H_1$ and $H_2$ are pseudohalfplanes such that $H_1\cap C$ is a subset of one of the intervals and $H_2\cap C$ is a subset of the other interval. Then $H_1\cap H_2=\emptyset$.
\end{lem}
\begin{proof}
	Notice that $\bar{H_1}$ and $\bar{H_2}$ are also in $\FF\cup \bar{\FF}$ and we can apply Lemma \ref{lem:twopshpcover} on them with the same $p$ and $q$. This implies that $\bar{H_1}\cup\bar{H_2}$ covers every vertex, which in turn implies that $H_1\cap H_2=\emptyset$.
\end{proof}

\subsection{Extremal vertices versus convex hull vertices}\label{sec:extremalvsextreme}

One might find it surprising that extremal vertices generalize points on the boundary of the convex hull and we do not have a definition that generalizes the notion of convex hull vertices. We argue why this cannot be done in a useful way (or how it can be done if needs be). First, if we could define the set of convex hull vertices of some pseudohalfplane hypergraph $\HH$ on vertex set $S$ in some way (and denote it by $CV$), then for this definition for being useful we would (arguably) need the following properties to hold:
\begin{enumerate}
	\item $CV\subseteq C$, that is, every convex hull vertex is an extremal vertex,
	\item $CV$ hits every hyperedge,
	\item if $\HH$ is defined on a point set by all halfplanes then $CV$ should be equal to the set of the geometric convex hull vertices,
	\item if $\HH\subset \HH'$ then $CV(\HH)\subseteq CV(\HH')$, that is, if a vertex is in $CV$ then it should remain in $CV$ if we add further hyperedges to $\HH$.
\end{enumerate}

Our primary aim is to generalize the geometric notion, so the third property is natural.
Also, the first two properties are quite natural to assume as these are true in the geometric setting as well. The fourth one is less evident but we think that it would be needed so that the definition is useful in practice. Namely, we know that $C$ hits every hyperedge and so we can take any minimal subset $CV$ of $C$ which hits every hyperedge (this is a minimal hitting set) and it is easy to see that this has the first three properties. On the other hand this is not well-defined. Indeed, take a halfplane hypergraph defined on a set of points $P$ in convex position and replace each vertex by a pair of vertices, we get a pseudohalfplane hypergraph $\HH_0$. Then any subset of vertices which has exactly one from each pair of vertices has the first two properties and the third does not apply to $\HH_0$. Which of these subsets should we choose as $CV$? We could,  e.g., choose the lexicographically smallest such set as $CV$ and then it is well-defined but seems to be quite arbitrary, we do not think that any other choice would be more natural. In particular the fourth property fails as for any deterministic choice of $CV$ in $\HH_0$ we can extend the hypergraph such that it is a halfplane hypergraph and exactly the other vertex from each pair becomes a convex hull vertex, thus we get a $\HH\subset \HH'$ such that $CV(\HH)\cap CV(\HH')=\emptyset$. 

However, there is one possible way to make the fourth property hold: if all minimal hitting subsets that are subsets of $C$ are considered to be a possible set of convex hull vertices. Indeed, each member of the family $\cal CV$ of such minimal hitting sets has the first three properties and $\cal CV$ also has the fourth property in the following sense: if $\HH\subset \HH'$ then if $CV'\in {{\cal CV}(\HH')}$, then there exists $CV\in {{\cal CV}(\HH)}$ s.t. $CV\subseteq CV'$. Indeed, $CV'$ being a hitting set of $\HH'$ implies that $CV'$ is also a hitting set of $\HH$ and so $CV'$ contains a minimal hitting set $CV$ of $\HH$. We also need that $CV\subseteq C(\HH)$ which holds as it is easy to see that by definition for the sets of extremal vertices we have $C(\HH')\subseteq C(\HH)$ which implies $CV\subseteq CV'\subseteq C(\HH')\subseteq C(\HH)$.

Nevertheless, unlike extremal vertices, we did not find such a notion of convex hull vertices useful for our purposes.

Finally, we remark that if we restrict our attention to maximal pseudohalfplane hypergraphs, that is, to hypergraphs to which no further hyperedge can be added without ruining the property that it is a pseudohalfplane hypergraph (with the given order of vertices), then for a maximal $\HH$ on vertex set $S$ we know that its extremal vertices are exactly those vertices $v$ for which $\{v\}$ and also $S\setminus \{v\}$ is a hyperedge of $\HH$ and so the unique minimal hitting set inside $C$ is $C$ itself and so the two notions (of extremal and extreme vertices) coincide.

\section{Proofs of discrete Helly-type theorems}\label{sec:proofpshp}

\begin{proof}[Proof of Theorem \ref{thm:primalaba}]
	We take a minimal hitting set $R$ of unskippable vertices, that is, $R$ is containment minimal for the property that it intersects every hyperedge. This exists as by Lemma \ref{lem:unskippable} every hyperedge contains an unskippable vertex. We show that $R$ has at most two vertices. Assume on the contrary that there exist $a<b<c$, all in $R$. As $R$ is minimal, there exist $A,B,C$ such that $a\in A$, $b\in B$, $c\in C$ are the only containments between these three vertices and three hyperedges. By the assumption there exists a vertex $s\in A\cap C$, wlog. $s<b$. Thus, $C$ contains $s$ and $c$ while it does not contain $b$, that is, $C$ skips $b$, contradicting that $b$ is unskippable.
\end{proof}

\begin{proof}[Proof of Theorem \ref{thm:primalpshp}]
	Given a family of pseudohalfplanes s.t. every pair of them intersects we need to find $3$ vertices that hit every pseudohalfplane.
	
	We take a minimal hitting set $R$ of extremal vertices (i.e., vertices of $C$), that is, $R$ is containment minimal for the property that it intersects every pseudohalfplane. This exists as by Claim \ref{claim:unskippablepshp} every hyperedge contains an extremal vertex. We claim that it contains at most $3$ vertices. Assume on the contrary that it contains at least $4$ different vertices, $p,q,r,s$, which appear in this order in the circular order. Now by the minimality of $S$ there exists a pseudohalfplane $H_1$ s.t. $p\in H_1$ yet $q,r,s\notin H_1$ and also a pseudohalfplane $H_2$ s.t. $r\in H_2$ yet $p,q,s\notin H_2$. We can apply Lemma \ref{lem:twopshpdisjoint} on $H_1,H_2,q,s$ to conclude that $H_1\cap H_2=\emptyset$, contradicting our assumption.	
\end{proof}

\begin{proof}[Proof of Theorem \ref{thm:dualpseudohp32}]
Our proof follows the steps of the proof of Theorem \ref{thm:dualhp1} \cite{jjr}, however, the geometric arguments used in \cite{jjr} need to be replaced by abstract counterparts. 

Consider a pseudohalfplane $H_1\in H$ that contains the largest number of vertices from $C$, the set of extremal vertices. 
If $H_1$ contains all vertices of $C$ then by Claim \ref{claim:everyconsecutive} $H_1$ contains all vertices, we are done.

Assume now that $H_1$ does not contain all vertices of $C$. By Lemma \ref{lem:hullinterval} $H_1$ intersects $C$ in an interval $I_1$ of the circular order defined on $C$. Let $p$ and $q$ be the two endvertices of this interval in the circular order of the extremal vertices ($p=q$ is possible). Let $r$ be an arbitrary point of $C\setminus H_1$. Using the assumption of the theorem there exists a pseudohalfplane $H_2$ that contains all of $p,q,r$ (if $p=q$ then we apply the assumption to $p,r$ and an arbitrary third vertex, which exists as $n\ge 3$). As $H_2\cap C$ is an interval $I_2$ containing $p$ and $q$, $I_2$ either contains $I_1$ plus also $r$, contradicting the maximality of $|H_1\cap C|$; or it contains $p$, $q$ and the interval $C\setminus I_1$. In this case by Lemma \ref{lem:twopshpcover} we get that $H_1$ and $H_2$ together cover every vertex, finishing the proof.
\end{proof}

We need the following statement, which is proved as part of the proof of Theorem 4.6 of \cite{abafree}:

\begin{lem}\cite{abafree}\label{lem:implicit}
	In a pseudohalfplane (resp. pseudohemisphere) hypergraph $\HH$ either there are $3$ (resp. $4$) hyperedges covering every vertex of $\HH$ or $\HH$ is the dual of a pseudohalfplane hypergraph.
\end{lem}

\begin{proof}[Proof of Theorem \ref{thm:dualpseudohp23}]
 Applying Lemma \ref{lem:implicit} to our $\HH$, in the former case we are done and in the latter case, we can apply Theorem \ref{thm:primalpshp} to the dual $\HH'$ of $\HH$, which is a pseudohalfplane hypergraph, to conclude that $3$ vertices of $\HH'$ hit every hyperedge of $\HH'$, which is equivalent to saying that $3$ hyperedges of $\HH$ cover the vertices of $\HH$.
\end{proof}

\begin{proof}[Proof of Theorem \ref{thm:dualpseudohsp24}]
	  Applying Lemma \ref{lem:implicit} to our $\HH$, in the former case we are done and in the latter case, we can apply Theorem \ref{thm:primalpshp} to the dual $\HH'$ of $\HH$, which is a pseudohalfplane hypergraph, to conclude that already $3$ vertices of $\HH'$ hit every hyperedge of $\HH'$, which is equivalent to saying that $3$ hyperedges of $\HH$ cover the vertices of $\HH$.
\end{proof}

\begin{proof}[Proof of Theorem \ref{thm:primalpshp32}]
	The proof follows a similar argument about halfplanes present in \cite{jjr}.	
	Given a pseudohalfplane hypergraph $\HH$ such that every triple of hyperedges has a common vertex, we need to prove that there exists a set of at most $2$  vertices that hits every hyperedge of $\HH$. If $n\le 2$ then we can take all vertices. Similarly, if we have at most two hyperedges then we can take a vertex from each of them and we are done. Otherwise, take the complement of every hyperedge to get the pseudohalfplane hypergraph $\bar{\HH}$. Assume on the contrary that no two vertices hit every hyperedge in $\HH$, then in $\bar{\HH}$ for every pair of points there is a hyperedge that contains both of them. Thus we can apply Theorem \ref{thm:dualpseudohp23} to conclude that there are $3$ hyperedges of $\bar{\HH}$ that together cover all the vertices. This means that there are $3$ hyperedges of $\HH$ that have no common vertex, contradicting our initial assumption. 
\end{proof}

\bigskip
To complement our upper bounds we show matching lower bound constructions. 

First we give some trivial examples showing that Theorem \ref{thm:primalaba} is optimal in every sense. We show that for ABA-free hypergraphs primal (and due to self-duality also dual) discrete Helly is not true with $k\rightarrow 1$ nor with $1\rightarrow k$ for any $k$. For the first, take a base set of size $l\ge k+1$ (with an arbitrary ordering) and take all size $l-1$ sets as hyperedges, this is $ABA$-free, every subfamily of $k$ hyperedges intersects in a vertex, yet there is no vertex which hits every hyperedge. For the second, take $k+1$ disjoint hyperedges of arbitrary size after each other, this is $ABA$-free, every hyperedge can be hit by a vertex (this property holds trivially for every hypergraph) yet no $k$ vertices hit all hyperedges.

These two constructions also imply that $1\rightarrow k$ and $k\rightarrow 1$ cannot hold for any $k$ for both the primal and dual case for pseudohalfplanes.

We have seen that primal discrete Helly holds with $2\rightarrow 3$ and with $3\rightarrow 2$. The following simple construction is a modification of a construction for halfplanes from \cite{jjr} and shows that it does not hold with $2\rightarrow 2$. Take $3k$ vertices $[0,3k-1]$ and for each $i\in [0,2]$ and $j\in [0,k-1]$ we take the set of size $k+1$ containing vertices $(ik,ik+1,\dots,ik+k-1)$ plus the vertex $(i+1)k+j$ as a hyperedge (indices are modulo $3k$). This hypergraph $\cH_0$ is easy to realize with halfplanes (see \cite{jjr}) and thus it is also a pseudohalfplane hypergraph. In this hypergraph every pair of hyperedges intersects yet no two vertices hit all the hyperedges. 

We have also seen that dual discrete Helly holds with $2\rightarrow 3$ and with $3\rightarrow 2$. The same construction as in the primal case shows that with $2\rightarrow 2$ it is not true. Indeed, in $\cH_0$ every pair of vertices is contained in some hyperedge yet no two hyperedges cover all vertices.

Thus for pseudohalfplanes we have covered all possible cases both in the primal and dual setting. 

For pseudohemisphere hypergraphs (for which the primal and dual cases are equivalent) we proved only that $4\rightarrow 2$ which leaves open several problems, we omit to list them all.

\section{Pseudohalfplanes versus halfplanes}\label{sec:vs}

Having spent this much effort to prove results about pseudohalfplane hypergraphs that are mostly already known for halfplane hypergraphs, it is worthwhile to investigate by what extent is the former a larger family compared to the latter. It is known that there are (simple) pseudoline arrangements that are non-stretchable already with $9$ pseudolines (based on the Pappus configuration), that is, which cannot be realized with line arrangements, moreover, almost all of them are such (see, e.g., Chapter 5 of \cite{handbook} by Felsner and Goodman). This suggests that pseudohalfplane hypergraphs are a much richer family than halfplane hypergraphs, but it might not be immediately obvious if there is a direct connection as arrangements encode geometric realizations while hypergraphs are strictly combinatorial structures. The aim of this section is to prove that the implication does hold. 

One can regard a pseudoline arrangement as a plane graph: the vertices of an arrangement of pseudolines are the intersection points of the pseudolines, the edges are the maximal connected parts of the pseudolines that do not contain a vertex and the faces are the maximal connected parts of the plane which are disjoint from the edges and the vertices of the arrangement.\footnote{The vertices of an arrangement should not be confused with the vertices of a hypergraph.}
We say that two pseudoline arrangements are (combinatorially) {\em equivalent} if there is a one-to-one adjacency-preserving correspondence between their pseudolines, vertices, edges and faces. We need the following:

\begin{thm}\label{thm:unique}
	Given a simple pseudoline arrangement $\AA$, let $P$ be a set of points which has exactly one point in each face of $\AA$. Let $\HH$ be a pseudohalfplane hypergraph whose vertex set is $P$ and for each pseudoline of $\AA$ it has a hyperedge which contains the points on one side of this pseudoline.\footnote{For each pseudoline we can choose the side arbitrarily. If for every pseudoline we choose the side above the pseudoline then $\HH$ is also ABA-free.} Then in every realization of $\HH$ with pseudohalfplanes the arrangement of the boundary pseudolines is equivalent to $\AA$.
\end{thm}

\begin{proof}
We claim that in any realization of $\cH$ by pseudohalfplanes, the arrangement of the boundary pseudolines is equivalent to $\AA$. To see this, take the arrangement $\AA'$ of the boundary pseudolines of any realization. In it we take the pseudolines in an arbitrary order: $l'_1,l'_2,\dots, l'_m$. The pseudolines in $\AA$ that are the boundaries of the hyperedges in the same order is denoted by $l_1,l_2,\dots, l_m$. Denote $\AA_{i}$ (resp. $\AA'_{i}$) the arrangement defined by the first $i$ boundary pseudolines of $\AA$ (resp. $\AA'$). 

The sub-arrangement $\AA'_1$ of $\AA'$ defined by $l'_1$ is unique (has two infinite faces with a bi-infinite curve separating them), and trivially equivalent to $\AA_1$. Assume now that for some $i$ we already know that $\AA'_{i-1}$ is equivalent to $\AA_{i-1}$ (and so we can identify their faces). We shall show that this holds also for $i$ instead of $i-1$. 

To see this we add $l_i'$, the $i$th pseudoline to $\AA'_{i-1}$.
The faces of the arrangement $\AA'_{i-1}$ partition the vertex set. Clearly, $l'_i$ must cross those faces of $\AA'_{i-1}$ for which the corresponding part of the partition is non-trivially intersected by the respective $i$th hyperedge (that is, the intersection of that part with the $i$th hyperedge is neither empty nor equal to that part), call these faces \emph{active}. Note that $l_i$ crosses in $\AA_{i-1}$ the active faces and no other face. Further, $l'_i$ must also cross the active faces in $\AA'_{i-1}$, although $l'_i$ may intersect further faces (we will see that this cannot happen).

\begin{figure}
	\begin{center}
		\includegraphics[height=6cm]{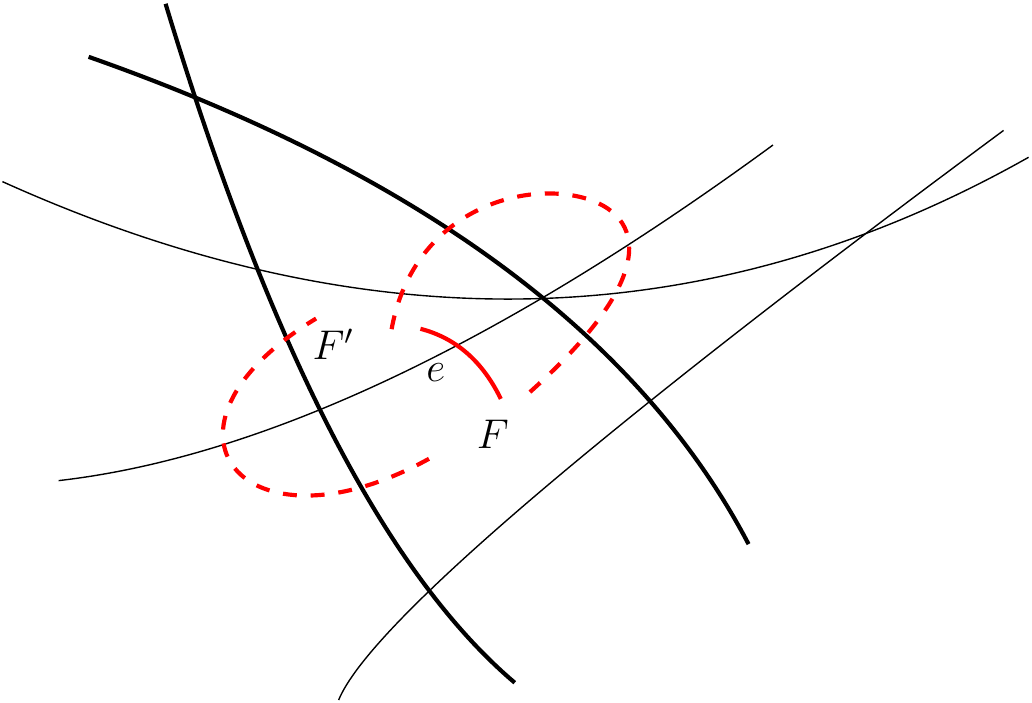}
		\caption{Proof of Theorem \ref{thm:unique}.}
		\label{fig:faces}
	\end{center}
\end{figure}

Let $F$ be an active face of the arrangement $\AA'_{i-1}$. 
We claim that topologically it is unique how $l'_i$ can cross $F$, and it is the same way as $l_i$ crosses $F$ in $\AA_{i}$. Notice that no two edges on the boundary of $F$ can be on the same pseudoline.
Let $F'$ be an active face neighboring $F$, i.e., it shares a common edge $e$ with $F$. We claim that $l_i'$ must cross $e$. Indeed, the (at most) two pseudolines which contain the edges that are consecutive with $e$ on the boundary of $F$ (see the bold pseudolines on Figure \ref{fig:faces}) separate the plane into at most $4$ parts. $F$ and $F'$ lie in the same part, thus a curve going from $F$ to $F'$ must cross both of these pseudolines an even number of times. As $l'_i$ crosses these pseudolines at most once, the only possibility is that a part of $l'_i$ connecting $F$ and $F'$ intersects both of these pseudolines zero times. The only way for this is if $l'_i$ crosses the edge $e$, as we claimed. After crossing $e$, $l'_i$ cannot cross again the pseudoline supporting $e$ and thus it cannot return to $F$. The same argument holds for $l_i$.

We call the $l_i$-order of the active faces the order in which $l_i$ crosses them. If $F$ is not the first nor the last active face in the $l_i$-order then our arguments so far show that both $l_i$ and $l'_i$ cross $F$ once and the same way. 

About the first and last face we only know that $l'_i$ leaves them once the same way as $l_i$. We are left to prove that these must be the two faces where $l'_i$ goes into infinity, the same way as $l_i$. Indeed, having determined the crossings on the boundaries of all the other active faces, we have found a crossing point for every pair of pseudolines (as in these faces $l'_i$ and $l_i$ have the same crossings and $l_i$ has no other crossings) and thus it cannot cross any other pseudoline anymore. That is, the first and last face in the $l_i$-order is also the first and last one in the order in which $l'_i$ crosses the faces, as claimed. 

Thus $\AA'_{i}$ is equivalent to $\AA_{i}$. Repeatedly applying this for $i=2,3,\dots m$ we get that the final arrangement $\AA'$ must be equivalent to $\AA$, finishing the proof.
\end{proof}

Note that in Theorem \ref{thm:unique} we need that the arrangement of the pseudolines is not loose, i.e., every pair of pseudolines intersects exactly once.

Now we can take an arbitrary non-stretchable simple pseudoline arrangement $\AA$. By Theorem \ref{thm:unique} we have an ABA-free hypergraph $\HH$ such that in every realization of $\HH$ with pseudohalfplanes the boundary pseudolines form an arrangement equivalent to $\AA$. Thus, $\HH$ cannot be realized with halfplanes as such a realization with halfplanes would also give the arrangement $\AA$, contradicting that $\AA$ was non-stretchable. 

\section{Chromatic number of pseudohalfplane hypergraphs}\label{sec:chromatic}

One may notice that our primal and dual results about pseudohalfplane hypergraphs look the same. A reason for this could be that the duals of pseudohalfplane hypergraphs are the same as pseudohalfplane hypergraphs. However, this is not the case, as shown already by a small example. This is a reason why we had to prove the primal and dual results separately. We note that the following claim is relevant already for the polychromatic coloring problems studied in \cite{abafree}, where the primal case is solved while in the dual the answer may still be the same as in the primal, but our best bounds are weaker.

\begin{claim}
	The family of pseudohalfplane hypergraphs and dual pseudohalfplane hypergraphs is not equal nor does contain one another.
\end{claim}

\begin{proof}
	It is easy to see that $K_4$ (the hypergraph on $4$ vertices containing all $6$ pairs as hyperedges) can be realized as a pseudohalfplane hypergraph yet it cannot be realized as a dual pseudohalfplane hypergraph \footnote{While this can be and was checked directly by a computer program, it also follows from the forthcoming Theorem \ref{thm:chrom} and thus we do not go into details about how such a program can be written.}. Thus the dual of $K_4$ can be realized as a dual pseudohalfplane hypergraph but not as a pseudohalfplane hypergraph, proving both containments. 
\end{proof}

In \cite{abafree} there was a systematic study of polychromatic problems but for some reason the chromatic number of the respective hypergraphs was not considered. Here we do this job, as it also gives us another proof of why $K_4$ (a $4$-chromatic hypergraph) cannot be realized as a dual pseudohalfplane hypergraph:

\begin{thm}\label{thm:chrom}
	The chromatic number of every ABA-free hypergraph is at most $3$, the chromatic number of every pseudohalfplane hypergraph is at most $4$, the chromatic number of every dual pseudohalfplane hypergraph is at most $3$ and these bounds are best possible.
\end{thm}

\begin{proof}
	The lower bounds are trivial as $K_3$ can be realized easily as an ABA-free hypergraph which is also a dual pseudohalfplane hypergraph by definition, also $K_4$ can be realized easily as a pseudohalfplane hypergraph as we can realize it already in the plane with $4$ points whose convex hull is a triangle and with appropriate $6$ halfplanes.
	
	We proceed with the upper bounds. For ABA-free hypergraphs we can alternately color the unskippable vertices with $2$ colors and use a third color for the skippable vertices. As every hyperedge intersects the unskippable vertices in an interval, it must be properly colored. In fact in \cite{abab} it was proved that ABAB-free hypergraphs (the definition see in \cite{abafree,abab}) have chromatic number at most $3$, which also implies proper $3$-colorability of ABA-free hypergraphs as every ABA-free hypergraph is also ABAB-free.
	
	For pseudohalfplane hypergraphs the upper bound follows from the more general result about the $4$-colorability of pseudo-disk wrt. pseudo-disk intersection hypergraphs \cite{Keszegh2020}. However, in our case there is a much simpler proof. Just take the extremal vertices in the order of the vertices, and for each of them if it is a topvertex (resp. bottomvertex) then we give a color different from the previous topvertex (resp. bottomvertex). With $3$ colors this can be done even when a vertex is both a top and a bottomvertex. Non-extremal vertices get the fourth color.	Then by Lemma \ref{lem:hullinterval} every hyperedge with at least two vertices either contains both an extremal and a non-extremal vertex or at least two vertices that are consecutive among the topvertices or bottomvertices, in every case the hypergraph is non-monochromatic.
	
	It remains to properly $3$-color dual pseudohalfplane hypergraphs. This is the most complicated part, the proof follows the idea of the respective earlier result of the author about $3$-coloring dual halfplane hypergraphs \cite{wcf2}. In \cite{abafree} it is proved that a hypergraph $\HH$ on an ordered set of vertices $S$ is a dual pseudohalfplane-hypergraph (that is, there exists a pseudohalfplane hypergraph whose dual is $\HH$) if and only if there exists a set $X\subset S$ and an ABA-free hypergraph $\cal F$ on $S$ such that the hyperedges of $\HH$ are the hyperedges $F\Delta X$ for every $F\in \cal F$ (where $\Delta$ denotes the symmetric difference of two sets). In the rest we assume this setup. First let $\FF_1$ (resp. $\FF_2$) be the subhypergraph of $\FF$ induced by $S\setminus X$ (resp. $X$), note that the vertex set of $\FF_1$ (resp. $\FF_2$) is $S\setminus X$ (resp. $X$). Every hyperedge $H=F\Delta X$ with at least $2$ vertices intersects $S\setminus X$ or $X$ in at least $2$ vertices or both of them in exactly one vertex. 
	
	We define a graph $G$ on $S$. The vertex set of $G$ is the set $U$ of the unskippable vertices of $\FF_1$ and $\FF_2$. First we connect two vertices  $v\in (S\setminus X)\cap U$ and $w\in X\cap U$ by an edge if there exists a hyperedge $H=\{v,w\}$ in $\HH$. We connect two vertices $v,w\in (S\setminus X)\cap U$ if they are consecutive (in the order of unskippable vertices) unskippable vertices of $\FF_1$. Finally, we connect two vertices $v,w\in X\cap U$ if they are consecutive unskippable vertices of $\FF_2$. We claim that this graph is outerplanar and thus $3$-colorable. The second and third set of edges form two paths that follow the order of the vertices. It is enough prove that in the first set of edges the connected vertices are in reversed order along the paths, i.e., there are no two edges $v_1w_1$ and $v_2w_2$ with $v_1,v_2\in S\setminus X$ and $w_1,w_2\in X$ such that $v_1<v_2$ and $w_1<w_2$. Showing this it follows that if we reverse one of the paths, the first set of edges forms a caterpillar between the two paths, thus the three parts together obviously form an outerplanar graph. Thus assume on the contrary that there are such two edges. This implies that there is a hyperedge $F_1=\{v_1\}\cup X\setminus\{w_1\}$ and $F_2=\{v_2\}\cup X\setminus\{w_2\}$ in $\FF$ (these are the hyperedges of $\FF$ with $\{v_1,w_1\}=F_1\Delta X$ and $\{v_2,w_2\}=F_2\Delta X$). If $w_1<v_1$ then these two hyperedges form an ABA-sequence on $w_1,v_1,w_2$, a contradiction. Otherwise, if $v_1<w_1$ then these two hyperedges form an ABA-sequence on $v_1,w_1,w_2$, again a contradiction. 
	
	Finally, we color the remaining vertices $S\setminus U$. By Lemma \ref{lem:assocunskippable} for each skippable vertex $v$ of $\FF_1$ there are $2$ unskippable vertices of $\FF_1$ such that every hyperedge that contains $v$ contains at least one of these. Thus color $v$ with a color different from the colors of these two unskippable vertices. We color similarly the skippable vertices of $\FF_2$. 
	
	We claim that the color we get is a proper $3$-coloring of $\HH$. First, if $H\in \HH$ contains a vertex from $S\setminus U$ then it is good by the last step of our coloring process. Otherwise, if $H$ contains at least two unksippable vertices of $\FF_1$ or of $\FF_2$ then we are done as then it contains two consecutive unskippable vertices in one of them, which get different colors as they are consecutive on one of the two paths we added to $G$. Finally, if $H$ contains exactly one unskippable vertex of $\FF_1$ and one of $\FF_2$ then these get different color as they were connected in $G$ in the first set of edges we added to $G$.
\end{proof}

Determining the maximal chromatic number of pseudohemisphere hypergraphs is an interesting open problem we leave open.

\section{Discussion}\label{sec:discussion}

We generalized several discrete Helly-type theorems about points and halfplanes to points and pseudohalfplanes, phrased equivalently as results about vertices and hyperedges of pseudohalfplane hypergraphs. While we have proved all possible results about ABA-free hypergraphs and pseudohalfplane hypergraphs, for pseudohemisphere hypergraphs we have shown only a result of type $4\rightarrow 2$, e.g., we do not know if a result of type $3\rightarrow l$ is true for some integer $l$.

Our discrete Helly-type results can be regarded as discrete variants of a Hadwiger Debrunner $(p,q)$-problem (see, e.g., the survey by Eckhoff \cite{Eckhoff2003}) in the special case when $p=q$. It would be interesting to consider also discrete Hadwiger Debrunner-type problems with $p\ne q$.

Pseudohalfplane hypergraphs are based on ABA-free hypergraphs. Similar to them, in \cite{abafree} ABAB-free (and ABABA-free etc.) hypergraphs were defined and in \cite{abab} it was shown that they are equivalent to hypergraphs defined on a point set by pseudodisks all containing the origin. Are there discrete Helly theorems about ABAB-free hypergraphs, or equivalently, about pseudodisks all containing the origin? This is especially interesting in the light of Theorem \ref{thm:convpseudodisks} about convex pseudodisks, which was not possible to generalize within the context of pseudohalfplane hypergraphs. On one hand, one important property of pseudohalfplane hypergraphs used in \cite{abafree}, that they always admit shallow hitting sets (for definitions see \cite{abafree}), does not always hold for ABAB-free hypergraphs. On the other hand, some other related positive results were proved in \cite{abab}.
\subsubsection*{Acknowledgement}

The author is grateful to D. P\'alv\"olgyi for the many discussions about these results and for the anonymous reviewers for their insightful comments, in particular for bringing to our attention the paper of Halman \cite{halman}.

\bibliographystyle{plainurl}
\bibliography{psdisk}

\end{document}